\documentclass[11pt, oneside]{article}   	
\usepackage{geometry}                		
\usepackage{graphicx}				
\usepackage{amssymb}
\usepackage{amsmath}
\usepackage{graphicx}				
\usepackage{amssymb}
\usepackage{amsmath}
\usepackage{amsthm}
\usepackage{harpoon}
\usepackage{accents}

\usepackage{tikz}  
\usepackage{float}
\restylefloat{table}
\usepackage{mathrsfs}
\usepackage{verbatim}

\usetikzlibrary{positioning,chains,fit,shapes,calc}  

\newtheorem{theorem}{Theorem}
\newtheorem{lemma}{Lemma}
\newtheorem{definition}{Definition}

\newtheorem{corollary}{Corollary}

\title{On a series for the upper incomplete Gamma function}
\author{Mario DeFranco}

\begin{document}

\maketitle

\abstract{We define an absolutely convergent series for the upper incomplete Gamma function $\Gamma(s,z)$ for $z\geq1$ and $s\in \mathbb{C}$. We express this series using certain polynomials which we define using the Stirling numbers of the first kind. We prove that these polynomials have positive coefficients by defining a three-parameter family of integers and certain linear operators on vector spaces of polynomials. We then apply this series to obtain a formula for the Riemann xi function valid at any $s \in \mathbb{C}$.}

\section{Introduction}

The Gamma function $\Gamma(s)$ is a special function important in many fields of mathematics, from statistics to number theory. L. Euler introduced $\Gamma(s)$
as a generalization of the factorial: for integer $n\geq0$, 
\[
\Gamma(n+1) = n!.
\]
For $s\in \mathbb{C}$, $\Gamma(s)$ may be defined by the integral 
\[
\Gamma(s) = \int_0^\infty e^{-u} u^{s-1} \, du
\]
with poles at the non-positive integers. See \cite{Davis} for a history of $\Gamma(s)$.

The upper incomplete Gamma function $\Gamma(s,z)$ defined as
\[
\Gamma(s,z)=\int_{z}^\infty e^{-u}u^{s-1}\, du
\]
was first investigated by F. E. Prym \cite{Prym} in 1877. See \cite{Jameson} for more information on $\Gamma(s,z)$. For $z > 0$ and $s$ real, we apply the change of variables $u \mapsto u z$ to obtain
\begin{equation} \label{int1}
\Gamma(s,z) = z^{s}\int_1 ^\infty e^{-u z} u^{s-1}\, du.
\end{equation}
The integral in \eqref{int1} appears in the definition of the Riemann xi function \cite{Riemann} in the following way: 
\begin{align} \label{xi}\nonumber
\xi(s) &= -\frac{s(1-s)}{2}\pi^{-\frac{s}{2}}\Gamma(\frac{s}{2})\zeta(s)\\
&= \frac{1}{2}(1-s(1-s)\sum_{n=1}^\infty \int_1^\infty e^{-\pi n^2 u} (u^{\frac{s}{2}-1}+u^{\frac{1-s}{2}-1} ) \, du).
\end{align}
In this paper we thus consider the integral 
\[
\int_1^\infty e^{-up} u^s \, du
\]
for $p\geq 1$ and $s \in \mathbb{C}$, and prove that it may be expressed as an absolutely convergent series
\[
\int_1^\infty e^{-up} u^s \, du = \frac{e^{-p}}{p}(1+ \sum_{n=1}^\infty \frac{g_n(s,p)}{(n+1)!})
\]
where $g_n(s,p)$ is a polynomial in $s$ and $p^{-1}$ defined below using the Stirling numbers of the first kind. For $p=1+x$, we prove that the two-variable polynomial $(1+x)^n g_n(s,1+x)$ has positive coefficients in $x$ and $s$. For the proof we define a three-parameter family of integers $A_{m,i}^k$ and certain linear operators on vector spaces of polynomials. We then apply this series to obtain a series for $\xi(s)$ absolutely convergent for any $s \in \mathbb{C}$. 

\section{The polynomials $g_n(s,p)$ }
To the integral 
\[
\int_1^\infty u^s e^{-u p}\, du
\]
we apply the change of variables 
\[
t = 1-e^{-up}
\]
to obtain
\[
\int_1^\infty u^s e^{-u p}\, du = \frac{1}{p^{s+1}}\int_{1-e^{-p}}^1 (-\log(1-t))^s \, dt.
\]
We thus consider the Taylor series of 
\[
(-\log(1-t))^s
\]
at $t=1-e^{-p}$ which we calculate in Lemma \ref{log derivative}.
We first recall the definition of the unsigned Stirling numbers of the first kind (see \cite{Graham}). 
\begin{definition} For integers $0 \leq i \leq n$, let
\[
 { n \brack i}
\]
denote the unsigned Stirling number of the first kind. These integers may be defined by the relation 
\[
{ n \brack i-1} +n { n \brack i} = { n +1 \brack i}
\]
and 
\[
{  0\brack n}={  n\brack 0}=\delta_{n,0}.
\]
\end{definition}

\begin{lemma}\label{log derivative}
For $n \geq 1$,
\[
(\frac{d}{dt})^n(-\log(1-t))^s = \frac{1}{(1-t)^n} \sum_{i=1}^{n}{ n \brack i} (s)_{i} (-\log(1-t))^{s-i}
\]
\end{lemma}
\begin{proof} 
We use induction on $n$. The statement is true for $n=1$. Assume it is true for some $n \geq 1$. Then by the induction hypothesis 
\[
(\frac{d}{dt})^{n+1}(-\log(1-t))^s = \frac{1}{(1-t)^{n+1}} \sum_{i=1}^{n+1}({ n \brack i-1} +n { n \brack i}) (s)_{i} (-\log(1-t))^{s-i}.
\]
The induction step now follows from the identity 
\[
{ n \brack i-1} +n { n \brack i} = { n+1 \brack i}.
\]
This completes the proof.
\end{proof}

We recall the notation for the falling factorial $(s)_i$:
\begin{definition}
For integer $i \geq 0$, let $(s)_i$ denote 
\[
(s)_i = \prod_{j=1}^i (s-j+1).
\]
\end{definition}
\begin{definition} 
For integer $n \geq 1$, define
\[
g_n(s,p) = \sum_{i=1}^n { n \brack i} \frac{(s)_i}{p^i}.
\]

\end{definition}

Now we have the standard identity 
\[
(s)_n = \sum_{i=1}^n (-1)^{n-i} {n \brack i } s^i \,\,\,\,\text{ with }\,\,\,\, (s)_0 = 1.
\]
Let $p=(1+x)$. Then in $(1+x)^n g_n(s,1+x)$, the coefficient of $x^js^{n-m}$ is  
\[
[ x^js^{n-m}] \big((1+x)^n g_n(s,1+x) \big)= \sum_{i=0}^m (-1)^i {m-i \choose j}{n \brack n-m+i} {n-m+i \brack n-m}.  
\]
For fixed $m$ and $i$, we consider the function of integer $n$ for $n \geq m$ 
\[
{n \brack n-m+i} {n-m+i \brack n-m}. 
\]
We consider the Newton series coefficients of this function next.
\subsection{The numbers $A_{m,i}^k$}

\begin{definition}
For integers $k \geq 0$ and fixed $m\geq i \geq 0$, we define the numbers $A_{m,i}^k \in \mathbb{Z}$  as the Newton series coefficients for $n \geq m$ 
\[
{n \brack n-m+i} {n-m+i \brack n-m} = \sum_{k=0}^{\infty} A_{m,i}^k {n \choose k}.
\]
\end{definition}

We prove the following recurrence relation for the  $A_{m,i}^k $. 

\begin{lemma} 
Let $0\leq i\leq m$ and $m \geq 1$.
 \begin{align*}
A_{m,i}^k &=(k-1)(A_{m-1,i}^{k-1} + A_{m-1,i}^{k-2}+ A_{m-1,i-1}^{k-2}) + (k+i-m-1)A_{m-1,i-1}^{k-1}\\
A_{0,0}^k &= \delta_{k,0} \\ 
A_{m,-1}^k&=0\\ 
A_{m,i}^{-1}&=A_{m,i}^{-2}=0.\\ 
\end{align*}
\end{lemma}
Furthermore for $m \geq 1$
\[
A_{m,i}^k=0
\]
if $k\leq m$ or $k \geq 2m+1$. 
\begin{proof} 
The boundary condition $\displaystyle A_{m,-1}^k=0$ follows from 
\[
{n-m-1 \brack n-m}=0
\]
for all $n-m\geq0$, and $\displaystyle A_{0,0}^k = \delta_{k,0}$ follows from 
\[
{n\brack n}=1
\]
for all $n \geq 0$. 
And we set $\displaystyle A_{m,i}^{-1}=A_{m,i}^{-2}=0$ because the Newton series requires $k \geq 0$. 

We will use the identities 
\begin{equation} \label{Stirling identity}
{n+1 \brack n+1-m} - {n\brack n-m} = n {n\brack n+1-m}
\end{equation}
and 
\begin{equation} \label{n binomial identity}
n {n \choose k} = k {n \choose k} +(k+1){n \choose k+1}.
\end{equation}

First, we use
\[
{n+1\choose k} - {n\choose k} ={n\choose k-1} 
\]
and obtain
\begin{align}\label{Amik}
&\sum_{k=0}^{\infty} A_{m,i}^k {n \choose k-1}\\ \label{Stirling difference}
=&{n+1 \brack n+1-m+i} {n+1-m+i \brack n+1-m} - {n \brack n-m+i} {n-m+i \brack n-m}. 
\end{align} 
Next we apply identity \eqref{Stirling identity} to line \eqref{Stirling difference}
and get
\begin{align*}
=& {n\brack n-m+i} {n+1-m+i \brack n+1-m} + n {n\brack n+1-m}  {n+1-m+i \brack n+1-m}   - {n \brack n-m+i} {n-m+i \brack n-m} \\
=& {n\brack n-m+i} ({n+1-m+i \brack n+1-m} -{n-m+i \brack n-m} ) + n {n\brack n+1-m+i}  {n+1-m+i \brack n+1-m}. \\ 
\end{align*}
We apply identity \eqref{Stirling identity} again and get
\begin{align*}
&= (n-m+i){n\brack n-m+i} {n-m+i \brack n-m+1}+ n {n\brack n+1-m+i}  {n+1-m+i \brack n+1-m}\\ 
&=(n-m+i)   (\sum_{k=0}^{\infty} A_{m-1,i-1}^k {n \choose k})+ n (\sum_{k=0}^{\infty} A_{m-1,i}^k {n \choose k}). \\
\end{align*}
Now we apply identity \eqref{n binomial identity} and get
\begin{align*} 
& = (-m+i)   (\sum_{k=0}^{\infty} A_{m-1,i-1}^k {n \choose k}) + \sum_{k=0}^\infty A_{m-1,i-1}^k (k{n \choose k}+ (k+1){n \choose k+1}) \\ 
& \,\, + \sum_{k=0}^{\infty} A_{m-1,i}^k(k{n \choose k}+ (k+1){n \choose k+1}) \\
\end{align*}
Thus equating the the coefficients of $\displaystyle {n \choose k-1}$ in the above sum and in line \eqref{Amik}, we get 
\begin{equation}\label{recurrence}
A_{m,i}^k =(k-1)(A_{m-1,i}^{k-1} + A_{m-1,i}^{k-2}+ A_{m-1,i-1}^{k-2}) + (k+i-m-1)A_{m-1,i-1}^{k-1}.
\end{equation}

We next show for $m \geq 1$ that 
\[
A_{m,i}^k=0
\]
if $k\leq m$ or $k \geq 2m+1$. We use induction on $m$. It is true for $m=1$ because 
\[
{n \brack n-1} = {n \choose 2}.
\]
 Assume it is true for some $m \geq 1$. Then the induction step follows from relation \eqref{recurrence}.

\end{proof}

\section{The positivity of $g_n(s,p)$}
By definition of $A_{m,i}^k$, we have for $n \geq 1$
\begin{align}\label{g coefficients}
 \big(1+x)^n g_n(s,1+x) &= s^n+\sum_{m=1}^n \sum_{j=0}^m s^{n-m}x^j \sum_{k=m+1}^{2m} {n \choose k} \sum_{i=0}^m (-1)^i {m-i \choose j}A_{m,i}^k 
\end{align}
where the coefficient of $s^n$ is $A_{0,0}^0 = 1$.
We thus show that 
\begin{equation} \label{j Amik}
\sum_{i=0}^m (-1)^i {m-i \choose j}A_{m,i}^k \geq 0.
\end{equation}
We express the quantity \eqref{j Amik} using certain linear operators defined next. 
\subsection{The operators $B_1,B_2,C_1$ and $C_2$}

\begin{definition}
 For each integer $k \geq 0$, let $t_k$ be an indeterminate.
 Define the linear operators $B_1, C_1: \mathbb{R}[t_{k-1}] \rightarrow \mathbb{R}[t_k] $ by 
 \begin{align*}
B_1(t_{k-1}^{n}) &= (k-1) t_{k}^n \\ 
C_1(t_{k-1}^n) &= -n t_{k}^{n+1}.
\end{align*}
 These operators $B_1$ and $C_1$ correspond  to going from the $A_{m-1,i}^{k-1}$ and $ A_{m-1,i-1}^{k-1}$, respectively, to $A_{m,i}^k$. The exponent $n$ corresponds to $k-m+i-1$.
 
 Define the linear operators $B_2, C_2: \mathbb{R}[t_{k-2}] \rightarrow \mathbb{R}[t_k] $ by 
 \begin{align*} 
B_2(t_{k-2}^n) &= (k-1)t_k^{n+1}\\
C_2(t_{k-2}^n) &= -(k-1)t_k^{n+2}. 
\end{align*}
These operators $B_2$ and $C_2$ correspond to going from the 
$ A_{m-1,i}^{k-2}$ and $ A_{m-1,i-1}^{k-2}$, respectively, to $A_{m,i}^k$.

For $m\geq 1$ and $m+1 \leq k \leq 2m$, define the polynomial $\mathcal{A}_{m,i}^k (t_k)$ to be the sum of all terms of the form
\[
Y_m..Y_1(t_0^0)
\]
where $Y_1$ is either $B_2$ or $C_2$, exactly $i$ of the $Y_h$ are either $C_1$ or $C_2$, and exactly $k-m$ of the $Y_h$ are either $B_2$ or $C_2$. Set $\mathcal{A}_{0,0}^{0}(t_0)=1$ and $A_{m,i}^{k}(t_k)=0$ otherwise. 

Let $\mathcal{A}_{m,i}^k (t)$ denote this polynomial evaluated at an arbitrary indeterminate $t$.

\end{definition}

\begin{lemma} \label{C1}
The polynomial $\mathcal{A}_{m-1,i-1}^{k-1} (t_{k-1})$ is an integer multiple of 
\[
t_{k-1}^{k-m+i-1}.
\]
\end{lemma}
\begin{proof} 
 Suppose the product of operators 
\[
Y_{m-1}..Y_1
\]
contributes a term to $\mathcal{A}_{m-1,i-1}^{k-1} (t_{k-1})$ and
let $\#B_l$ and $\#C_l$ for $l=1,2$ denote the number of times that operator appears in the product. By definition, the operator $B_1$ preserves the exponent of the indeterminate; $B_2$ and $C_1$ increment the exponent by 1; and $C_2$ increments it by $2$. Thus 
\[
Y_{m-1}..Y_1(t_0^0)
\]
is an integer multiple of 
\[
t_{k-1}^{\#B_2 + \#C_1 + 2\#C_2}
\]
and we have
\begin{align*}
   \#C_1+\#C_2 &= i-1\\
   \#B_2+\#C_2 &=(k-1)-(m-1) 
\end{align*}
by the assumption that the product contributes a term to $\mathcal{A}_{m-1,i-1}^{k-1} (t_{k-1})$. 
These equations imply 
\[
\#B_2 + \#C_1 + 2\#C_2 = k-m+i-1.
\]
This completes the proof.
\end{proof}

\begin{lemma} \label{mathcalA relation}
\[
\mathcal{A}_{m,i}^k(t_k) = (k-1)( \mathcal{A}_{m-1,i}^{k-1}(t_k) -t_k^2\mathcal{A}_{m-1,i-1}^{k-2}(t_k) +t_k\mathcal{A}_{m-1,i}^{k-2}(t_k)) -(k-m+i-1)t_k\mathcal{A}_{m-1,i-1}^{k-1}(t_k). 
\]
\end{lemma}
\begin{proof} 
By definition 
\[
\mathcal{A}_{m,i}^k(t_k) = B_1( \mathcal{A}_{m-1,i}^{k-1}(t_{k-1})) + C_2(\mathcal{A}_{m-1,i-1}^{k-2}(t_{k-2})) +B_2(\mathcal{A}_{m-1,i}^{k-2}(t_{k-2}))+C_1(\mathcal{A}_{m-1,i-1}^{k-1}(t_{k-1})). 
\]
By definition of the operators, 
we have 
\begin{align*}
  B_1( \mathcal{A}_{m-1,i}^{k-1}(t_{k-1})) = (k-1)\mathcal{A}_{m-1,i}^{k-1}(t_{k}) \\ 
  C_2(\mathcal{A}_{m-1,i-1}^{k-2}(t_{k-2})) = -(k-1)t_{k}^2\mathcal{A}_{m-1,i-1}^{k-2} (t_{k})\\ 
  B_2(\mathcal{A}_{m-1,i}^{k-2}(t_{k-2})) = (k-1)t_{k}\mathcal{A}_{m-1,i}^{k-2}(t_{k}),
\end{align*}
and by Lemma \ref{C1} we have
\[
C_1(\mathcal{A}_{m-1,i-1}^{k-1}(t_{k-1})) = -(k-m+i-1)t_{k}\mathcal{A}_{m-1,i-1}^{k-1}(t_k).
\]
This completes the proof.
\end{proof}
\begin{corollary} \label{A poly A}
\[
(-1)^iA_{m,i}^k = \mathcal{A}_{m,i}^k(1)
\]
\end{corollary}
\begin{proof}
From Lemma \ref{mathcalA relation} the numbers $(-1)^i\mathcal{A}_{m,i}^k(1)$ satisfy the same recurrence relation as the $A_{m,i}^k$.
\end{proof}

\begin{definition}
Define $f(k,M,N) \in \mathbb{R}[t_k]$ to be
\[
f(k,M,N) = t_{k}^M(1-t_k)^N. 
\]
If $k,M,$ and $N\geq0$, and
 \[
k-M-N\geq 0,
\]
then we say that such an $f(k,M,N)$ satisfies the positivity condition.

 We say that an operator $L$ preserves the positivity condition if $L$ maps an $f(k,M,N)$ that satisfies the positivity condition to a positive-coefficient linear combination of terms $f(k',M',N')$ that each satisfy the positivity condition.
\end{definition}

\begin{lemma}\label{preserve}
The operators 
\[
B_1+C_1, \,\, B_2 + C_2, \, \,\, B_1,\, \, \text{ and } B_2 
\]
each preserve the positivity condition.
\end{lemma}
\begin{proof}
We use 
\[
C_1(t_k^n) = \left(-t_k^2\frac{d}{dt_k}( t_k^n)\right) |_{t_k \mapsto t_{k+1} }. 
\]
Then we calculate
\begin{align*}
(B_1+C_1)(t_k^M (1-t_k)^N) = & \,(k-M)f(k+1,M,N) + Mf(k+1,M,N+1) \\ 
&+ Nf(k+1,M+2,N-1).
\end{align*}
Thus if $f(k,MN)$ satisfies the positivity condition, each of the terms on the right have non-negative coefficients and also satisfy the positivity condition, except for possibly 
\[
Nf(k+1,M+2,N-1)
\] 
which may have $N-1<0$. But if $N\geq0$ and $N-1<0$, then $N=0$ and this term has 0 coefficient.

We next calculate
\begin{align*}
(B_2+C_2)(t_k^M (1-t_k)^N) = & \,(k+1)f(k+2,M+1,N+1),\\
\end{align*}
\[
B_1(f(k,M,N)) = kf(k+1,M,N),
\]
and 
\[
B_2(f(k,M,N)) = (k+1)f(k+2,M+1,N).
\]


This completes the proof.
\end{proof}

\begin{lemma} \label{BCB}
For integers $(m+1) \geq 1$, $(m+1)+1 \leq k \leq 2(m+1)$, and $0 \leq j \leq (m+1)$,
\begin{align*}
&\sum_{l=1}^2 \left((B_l+C_l)(\sum_{i=0}^m {m-i \choose j}\mathcal{A}_{m,i}^{k-l}(t_{k-l}) ) + B_l( \sum_{i=0}^m{m-i \choose j-1}\mathcal{A}_{m,i}^{k-l}(t_{k-l})) \right)\\
=& \sum_{i=0}^{m+1}{m+1-i \choose j}\mathcal{A}_{m+1,i}^k(t_k).
\end{align*}
\end{lemma}
\begin{proof}
From the definition of the operators and $\mathcal{A}_{m,i}^{k}(t_{k})$, we have 
\begin{equation}\label{i=0}
B_1(\mathcal{A}_{m,0}^{k-1}(t_{k-1}))+ B_2(\mathcal{A}_{m,0}^{k-2}(t_{k-2})) = \mathcal{A}_{m+1,0}^{k}(t_{k}),
\end{equation}
\begin{equation}\label{0<i<m+1}
B_1(\mathcal{A}_{m,i}^{k-1}(t_{k-1}))+ B_2(\mathcal{A}_{m,i}^{k-2}(t_{k-2}))+C_1(\mathcal{A}_{m,i-1}^{k-1}(t_{k-1}))+ C_2(\mathcal{A}_{m,i-1}^{k-2}(t_{k-2})) = \mathcal{A}_{m+1,i}^{k}(t_{k}),
\end{equation}
and 
\begin{equation}\label{i=m+1}
C_1(\mathcal{A}_{m,m}^{k-1}(t_{k-1}))+ C_2(\mathcal{A}_{m,m}^{k-2}(t_{k-2})) = \mathcal{A}_{m+1,m+1}^{k}(t_{k}).
\end{equation}
Now 
\begin{equation} \label{binomial j}
{N \choose j}+ {N \choose j-1}  = {N+1 \choose j}
\end{equation}
for all integers $N,j\geq0$. 

From equations \eqref{i=0} and \eqref{binomial j}, it follows that 
\[
\sum_{l=1}^2 B_l({m \choose j}\mathcal{A}_{m,0}^{k-l}(t_{k-l}))+ B_l({m \choose j-1}\mathcal{A}_{m,0}^{k-l}(t_{k-l}))  = {m+1 \choose j} \mathcal{A}_{m+1,0}^{k}(t_{k}).
\]
This gives the $i=0$ term on the right side of the lemma statement.

From equations \eqref{0<i<m+1} and \eqref{binomial j}, for $1\leq i\leq m$, it follows that 
\begin{align*}
&\sum_{l=1}^2 B_l({m -i\choose j}\mathcal{A}_{m,i}^{k-l}(t_{k-l}))+ B_l({m-i \choose j-1}\mathcal{A}_{m,i}^{k-l}(t_{k-l})) \\
+&\sum_{l=1}^2 C_l({m-(i-1) \choose j}\mathcal{A}_{m,i-1}^{k-l}(t_{k-l}))\\ 
= &{m+1-i \choose j}\mathcal{A}_{m+1,i}^k(t_k).
\end{align*}
This gives the $i$-th term on the right side of the lemma statement for $1 \leq i \leq m$.

From equations \eqref{i=m+1}, it follows that 
\[
\sum_{l=1}^2 C_l({m -m \choose j}\mathcal{A}_{m,m}^{k-l}(t_{k-l})) = {m+1 -(m+1) \choose j} \mathcal{A}_{m+1,m+1}^{k}(t_{k}).
\]
This gives the $i=m+1$ term on the right side of the lemma statement. This completes the proof.

\end{proof}

\begin{theorem}\label{A alt positivity}
For $m \geq 1$; $m+1\leq k\leq 2m$; and $0\leq j \leq m$, 
\[
\sum_{i=0}^m (-1)^i {m-i \choose j} A_{i,m}^k \geq 0.
\]
\end{theorem}
\begin{proof}
We prove that the polynomial 
\[
\sum_{i=0}^m  {m-i \choose j} \mathcal{A}_{m,i}^k(t_k)
\]
is a linear combination with positive coefficients of terms of the form $f(k,M,N)$ that satisfy the positivity condition. Then the theorem follows from this fact and 
\[
(-1)^iA_{m,i}^k = \mathcal{A}_{m,i}^k(1)
\]
from Corollary \ref{A poly A}.

We use induction on $m$. The statement for $m=1$ follows from Lemma \ref{BCB}; when $j=0$, we have 
\[
\sum_{i=0}^1  {1-i \choose 0} \mathcal{A}_{1,i}^2(t_2) = f(2,1,1)
\]
and when $j=1$ we have 
\[
\sum_{i=0}^1  {1-i \choose 1} \mathcal{A}_{1,i}^2(t_2) = f(2,1,0).
\]
We this assume that the statement is true for some $m$ and all $j$ with $0 \leq j \leq m$. We use Lemma \ref{BCB} to show that the polynomial
\[
\sum_{i=0}^{m+1}  {m+1-i \choose j} \mathcal{A}_{m+1,i}^k(t_k)
\] 
is a a positive-coefficient linear combination of $f(k,M,N)$ that satisfy the positivity condition in the following cases for $j$. 

For $j=0$, we have from Lemma \ref{BCB}
\[
\sum_{l=1}^2 (B_l+C_l)(\sum_{i=0}^m {m-i \choose 0}\mathcal{A}_{m,i}^{k-l}(t_{k-l}) );
\]
the polynomial 
\[
\sum_{i=0}^m {m-i \choose 0}\mathcal{A}_{m,i}^{k-l}(t_{k-l}) 
\]
is a positive-coefficient linear combination of terms that satisfy the positivity condition by the induction hypothesis, and the operators $B_1+C_1$ and $B_2+C_2$ preserve the positivity condition.

For $1\leq j\leq m$, we have from Lemma \ref{BCB}
\[
\sum_{l=1}^2 (B_l+C_l)(\sum_{i=0}^m {m-i \choose j}\mathcal{A}_{m,i}^{k-l}(t_{k-l}) ) + B_l( \sum_{i=0}^m{m-i \choose j-1}\mathcal{A}_{m,i}^{k-l}(t_{k-l})) ;
\]
each polynomial
\[
\sum_{i=0}^m {m-i \choose j}\mathcal{A}_{m,i}^{k-l}(t_{k-l}) \text{ and }  \sum_{i=0}^m{m-i \choose j-1}\mathcal{A}_{m,i}^{k-l}(t_{k-l})
\]
 is a positive-coefficient linear combination of terms that satisfy the positivity condition by the induction hypothesis, and the operators $B_1+C_1$, $B_2+C_2$, $B_1$ and $B_2$ preserve the positivity condition.

For $j =m+1$, we have from Lemma \ref{BCB}
\[
 \sum_{l=1}^2 B_l( \sum_{i=0}^m{m-i \choose m}\mathcal{A}_{m,i}^{k-l}(t_{k-l})) ;
\]
the polynomial 
\[
\sum_{i=0}^m{m-i \choose m}\mathcal{A}_{m,i}^{k-l}(t_{k-l})
\]
is a positive-coefficient linear combination of terms that satisfy the positivity condition by the induction hypothesis, and the operators $B_1$ and $B_2$ preserve the positivity condition.
 This proves the induction step and completes the proof.
\end{proof}

\begin{theorem} \label{g positivity}
As a polynomial in $s$ and $x$,
\[
(1+x)^n g_n(s,1+x)
\]
has positive coefficients.
\end{theorem}
\begin{proof}
The theorem follows from equation \eqref{g coefficients} and Theorem \ref{A alt positivity}.
\end{proof}

\begin{theorem} \label{series convergent}
Let 
\[
g_n(s,p) = \sum_{i=1}^{n} {n \brack i} \frac{(s)_i}{p^i}.
\]
Then for $p\geq 1$ and $s \in \mathbb{C}$, 
\[
\int_1^\infty e^{-u p}u^s \, du = \frac{e^{-p}}{p}(1+\sum_{n=1}^\infty \frac{g_n(s,p)}{(n+1)!})
\]
and the sum is absolutely convergent. 
\end{theorem}
\begin{proof}
By Lemma \ref{log derivative}, we have for $1-e^{-p} \leq t < 1$
\[
(-\log(1-t))^s = p^s(1+  \sum_{n=1}^\infty e^{n p} \frac{(t - (1-e^{-p}))^n}{n!} g_n(s,p) ).
\]
Now 
\begin{align} \nonumber
\int_1^\infty e^{-u p} u^s \, du  &= p^{-s-1}\int_{1-e^{-p}}^1 (-\log(1-t)^s \, dt \\ \label{int and sum}
& =p^{-s-1}\int_{1-e^{-p}}^1 p^s(1+  \sum_{n=1}^\infty e^{n p} \frac{(t - (1-e^{-p})^n}{n!} g_n(s,p) ) \, dt
\end{align}
If $s$ is positive and $p\geq 1$, then by Theorem \ref{g positivity} each $g_n(s,p) \geq 0$, and by Tonelli's theorem we may interchange the integration and summation at line \eqref{int and sum} to obtain
\begin{equation}
\int_1^\infty e^{-u p} u^s \, du=\frac{e^{-p}}{p}(1+  \sum_{n=1}^\infty \frac{g_n(s,p)}{(n+1)!}  ).
\end{equation}
Now suppose $s \in \mathbb{C}$. It also follows from Theorem \ref{g positivity} that
\[
|g_n(s,p)| \leq g_n(|s|,p).
\]
Then
\begin{align*}
|\sum_{n=1}^\infty \frac{g_n(s,p)}{(n+1)!}|  &\leq  \sum_{n=1}^\infty \frac{|g_n(s,p)|}{(n+1)!}\\ 
& \leq  1+\sum_{n=1}^\infty \frac{g_n(|s|,p)}{(n+1)!} \\
& \leq  \int_1^\infty e^{-up} u^{|s|} \, du
\end{align*}
Thus 
\[
 \sum_{n=1}^\infty \frac{g_n(s,p)}{(n+1)!}
\]
is absolutely convergent for any $s \in \mathbb{C}$. Now, a series of polynomials with positive coefficients that is absolutely convergent on $\mathbb{C}$ is holomorphic on $\mathbb{C}$ by the standard theorem that a a sequence of holomorphic functions that is uniformly convergent on compacts subsets of a domain $\Omega$ is holomorphic on $\Omega$ (see Stieltjes \cite{Stieltjes} and \cite{Stein} for more information). Thus for fixed $p\geq1$
\[
\frac{e^{-p}}{p}(1+  \sum_{n=1}^\infty \frac{g_n(s,p)}{(n+1)!}  )
\]
is entire in $s$ and agrees with 
\[
\int_1^\infty e^{-up} u^s \, du
\]
when $s$ is positive and real. Therefore these two functions must agree for $s$ on $\mathbb{C}$. 
\end{proof}

\section{Application to the Riemann xi function}

From the definition of the Riemann xi function $\xi(s)$ in equation \eqref{xi}, we have 
\[
\xi(s) = \frac{1}{2}(1 - s(1-s)\sum_{n=1}^\infty (\int_1^\infty e^{-\pi n^2} u^{\frac{1-s}{2}-1}\, du +\int_1^\infty e^{-\pi n^2} u^{\frac{s}{2}-1}\, du)).
\] 
For integer $k \geq 1$, let $b_k$ denote the constants
\[
b_k = \sum_{n=1}^\infty \frac{e^{-\pi n^2}}{(\pi n^2)^k}. 
\]
We then apply Theorem \ref{series convergent} to obtain
\begin{multline}
\pi^{-\frac{s}{2}}\Gamma(\frac{s}{2})\zeta(s) = 2b_1 + \sum_{n=1}^\infty \frac{1}{(n+1)!} \sum_{i=1}^{n} b_{i+1} {n \brack i} \left((\prod_{j=1}^{i} \frac{1-s}{2}-j)+(\prod_{j=1}^{i} \frac{s}{2}-j)\right) \\ 
-(\frac{1}{1-s}+ \frac{1}{s})
\end{multline}
where the rearrangement of the sum is justified by the absolute convergence of the series from Theorem \ref{series convergent}.
We may re-write this as 
\begin{multline}
-s(1-s)\pi^{-\frac{s}{2}}\Gamma(\frac{s}{2})\zeta(s) = 1 - 2\sum_{n=0}^\infty \frac{1}{(n+1)!} \sum_{i=0}^{n} b_{i+1} {n \brack i} \left(s(\prod_{j=0}^{i} \frac{1-s}{2}-j)+(1-s)(\prod_{j=0}^{i} \frac{s}{2}-j)\right).  
\end{multline}

\section{Further Work}

\begin{itemize}

\item See if the sum of of all the $s^{n-m}x^j$ can be re-arranged by grouping together the positive terms and indexing them by paths in a graph. Perhaps the constants 
\[
\sum_{n=1}^\infty \frac{e^{-\pi n^2}}{(\pi n ^2)^k}
\]  
then become different constants which can be analyzed by Poisson summation or something else. 

\item Apply these formulas to $\xi(s)$ using a different integral kernel such the Polya-Schur kernel. 

\item See how the proofs for Jensen polynomials for $\xi$ of degree 2 and 3 for having real zeros can be expressed using these formulas for $a_k$.

\item See if these formulas have $q$-analogues. 

\item See if series can be obtained using the formula
\[
\int_1^\infty u^s e^{-up}\, du=\sum_{n=0}^\infty (2^n)^{s+1}  e^{-2^n p} \int_{0}^1 (1+u)^s e^{-2^n u p} \, du.
\]

\item See if there are more quantitative bounds to prove absolute convergence of the series.

\end{itemize}


\begin{thebibliography}{9}

\bibitem{Davis} Davis, P. J. ``Leonhard Euler's Integral: A Historical Profile of the Gamma Function". American Mathematical Monthly. 66 (10): (1959). pp. 849-869. doi:10.2307/2309786. JSTOR 2309786. Retrieved 3 December 2016

\bibitem{Graham} Graham, Ronald L.; Knuth, Donald E.; Patashnik, Oren (February 1994). Concrete Mathematics - A foundation for computer science (2nd ed.). Reading, MA, USA: Addison-Wesley Professional. pp. xiv+657. ISBN 0-201-55802-5. MR 1397498. 

\bibitem{Jameson} Jameson, G.J.O. ``The incomplete gamma functions". DOI: https://doi.org/10.1017/mag.2016.67 Published online by Cambridge University Press: 14 June 2016

\bibitem{Prym} Prym, F. E. ``Zur Theorie der Gammafunction. J. Reine Angew. Math. 82, (1877). pp. 165-172. 

\bibitem{Riemann}  Riemann, Bernhard. ``\"{U}ber die Anzahl der Primzahlen unter einer gegebenen Gr\"{o}sse," Monatsberichte der Berliner Akademie. (1859). In Gesammelte Werke, Teubner, Leipzig (1892), Reprinted by Dover, New York (1953) 

\bibitem{Stein} Shakarchi, Rami and Stein, Elias M. Princeton Lectures in Analysis: II Complex Analysis. Princeton University Press, 2003

\bibitem{Stieltjes} Stieltjes, T. J. ``Recherches sur les fractions continues", Ann. Fac. Sci. Toulouse VIIIJ (1894), pp. 1-122.



\end{thebibliography}
\end{document}